\documentclass{article}
\usepackage{amsmath, amsthm}
\usepackage{amssymb}
\newtheoremstyle{theorem}%name
  {10pt}          % space above
  {10pt}  % space below
  {\sl}  % bofy font
  {\parindent}     % ident - empty=no indent,  \parindent= paragraph indent
  {\bf}  % thm head font
  {. }    % punctuation after thm head
  { }    % space after thm head: `` ``=normal \newline=linebreak
  {}     % thm head specification
\theoremstyle{theorem}
\newtheorem{theorem}{Theorem}

\newtheorem{proposition}[theorem]{Proposition}
\newtheorem{example}[theorem]{Example}
\newtheoremstyle{defi}%name
  {10pt}          % space above
  {10pt}  % space below
  {\rm}  % bofy font
  {\parindent}     % ident - empty=no indent,  \parindent= paragraph indent
  {\bf}  % thm head font
  {. }    % punctuation after thm head
  { }    % space after thm head: `` ``=normal \newline=linebreak
  {}     % thm head specification
\theoremstyle{defi}
\newtheorem{definition}[theorem]{Definition}

\begin{document}

\title{\large\bf  On Multiplicative Fractional Calculus}
\author{\small \bf Thabet Abdeljawad \\  Mathematics and General Sciences Department \\ Prince Sultan University,  \\11586 , Riyadh, Kingdom of Saudi Arabia}
\maketitle

\begin{abstract}We set the main concepts for multiplicative fractional calculus. We define Caputo, Riemann and Letnikov multiplicative fractional derivatives and multiplicative fractional integrals and study some of their properties. Finally, the multiplicative analogue of the local conformable fractional derivative and integral is studied.

 {\bf AMS Subject Classification: } .


\end{abstract}

\section{Introduction} \label{s:1}

Multiplicative calculus has a relatively restrictive area of applications than the calculus of Newton and Leibnitz. Indeed, it covers only positive functions. However, the necessity of developing and applying multiplicative calculus resembles the importance of the polar coordinates while rectangular systems already exist. Also, it is of our thought that multiplicative calculus is a useful mathematical tool for economics and finance and other branches of science because several interpretations are given by means of multiplicative derivatives where the logarithmic scale appears. Actually, many of the scientific tables are given in logarithmic scales. For example, the level of sound signals, the acidities of chemicals, and the magnitude earthquakes are all given in logarithmic scales. In fact, an important work of Benford \cite{Benford}that was published in 1938 described  many physical quantities in the nature that are exponentially varying type.Indeed, Benford made a statistical work on twenty different tables containing about twenty thousand quantities,which had been taken from the real world. He observed that such quantities representing the surface area of rivers, the population figures of countries, the electricity bills, the stock prices in markets, etc. are all of exponential type varying \cite{Tarantola}.

The multiplicative calculus, which mainly deals with exponentially varying functions, was firstly proposed by mathematical biologists Volterra and Hostinsky \cite{Volterra} in 1938. For more works later about multiplicative calculus we refer to \cite{Grossman1,Grossman2, Stanley, Bashirov, Ali}.
An interesting application where the theory  of multiplicative calculus cab help more than Newtonian Calculus can be found in \cite{Bashirov}. The authors there, used  multiplicative calculus to prove easier that that the Cauchy's function is infinitely many times differentiable on $\mathbb{R}$ but not analytic.

Fractional calculus, the calculus of differentiation and integration of arbitrary order, started to attract scientists, in the last few decades, in different fields of engineering and science extensively with variety of applications \cite{Pod, Fat}. Moreover, the discrete counterpart of fractional calculus was the field of attention for many researchers in the last two decades (see \cite{Thabet} and the references therein).

In this work, we bring together multiplicative calculus and fractional calculus. The basic types of fractional integrals and derivatives will be defined and described to understand how exponential varying quantities look-like with arbitrary order reflected by delay effects.

\section{Short Preliminaries: Multiplicative Calculus and Fractional Calculus  } \label{s:2}
\begin{definition}
Let $f(x)$ be a positive real valued function. Then, the (forward) multiplicative derivative of $f(x)$ is defined by
\begin{equation}\label{d1}
  \frac{d^*}{dx^*}f(x)=f^*(x)=\lim_{h\rightarrow 0}(\frac{f(x+h)}{f(x)})^{1/h},~~~f(x) \in \mathbb{R}^+
\end{equation}

\end{definition}
Now, for our purposes in fractionalizing the multiplicative derivative we define the (backward)multiplicative derivative.
\begin{definition}
Let $f(x)$ be a positive real valued function. Then, the (backward) multiplicative derivative of $f(x)$ is defined by
\begin{equation}\label{d2}
  \frac{d_*}{dx_*}f(x)=f_*(x)=\lim_{h\rightarrow 0}(\frac{f(x)}{f(x-h)})^{1/h},~~~f(x) \in \mathbb{R}^+
\end{equation}
Inductively, it can be shown that $$ f^*~ ^{(n)}(x)=f_*~ ^{(n)}(x)=exp(\frac{d^n}{dx^n}ln(f(x)))$$
\end{definition}
Both the backward and forward multiplicative derivatives are related to the usual derivative by
\begin{equation}\label{u}
   \frac{d^*}{dx^*}f(x)=\frac{d_*}{dx_*}f(x)= exp(\frac{d}{dx}ln f(x))=exp(\frac{f^\prime(x)}{f(x)}).
\end{equation}
This can be shown easily either starting from the definition or by applying the natural logarithm to both sides in the definition and taking the limit to both sides and making use of that $f(x)$ is differentiable by the help of L'Hopital rule.
For the sake of being convenient in the multiplicative fractional cases we shall write $(\ominus \frac{d_*}{dx_*})^n f(x)=(\ominus f_*)^{(n)}(x)=e^{(-\frac{d}{dx})^n(ln\circ f)(x)}$.

The forward or backward multiplicative integrals are given by
\begin{equation}\label{m1}
  \int_a^b f(x)^{dx}=\int_a^b f(x)_{dx}=exp (\int_a^b lnf(x)dx).
\end{equation}
Of course the backward and forward multiplicative integrals may differ when the time scale is different from $\mathbb{R}$.

Let $[a,b]$ be a finite interval in $\mathbb{R}$. The left Riemann-Liouville fractional integral $(~_{a}I^{\alpha}f)(x)$ of order $\alpha \in \mathbb{C}$, $Re(\alpha)>0$ starting from $a$ is defined by
\begin{equation}\label{r1}
 ( ~_{a}I^{\alpha}f)(x)=\frac{1}{\Gamma(\alpha)}\int_a^x (x-t)^{\alpha-1}f(t)dt,~~~x>a.
\end{equation}

and the right Riemann-Liouville fractional integral by
\begin{equation}\label{r2}
 ( I_b^{\alpha}f)(x)=\frac{1}{\Gamma(\alpha)}\int_x^b (t-x)^{\alpha-1}f(t)dt,~~~~x<b.
\end{equation}
For $\alpha=n \in \mathbb{N}$, equations (\ref{r1}) and (\ref{r2}) coincide (by means of Cauchy formula) with the $nth$ integrals of the form

\begin{eqnarray}
 \nonumber
  ( ~_{a}I^nf)(x) &=& \int_a^x dt_1 \int_a^{t_1}dt_2...\int_a^{t_{n-1}} f(t_n) dt_n\\
   &=& \frac{1}{(n-1)!}\int_a^x (x-t)^{(n-1)}f(t)dt,~~~x>a
\end{eqnarray}
and

\begin{eqnarray}
 \nonumber
  ( I_b^nf)(x) &=& \int_x^b dt_1 \int_{t_1}^b dt_2...\int_{t_{n-1}}^b f(t_n) dt_n \\
   &=&\frac{1}{(n-1)!}\int_x^b (t-x)^{n-1}f(t)dt,~~~~x<b.
\end{eqnarray}

The left and right Riemann-Liouville fractional derivatives of order $\alpha,~~Re(\alpha)>0$ are defined by

\begin{equation}\label{b1}
 ( _{a}D^{\alpha}f)(x)= \frac{d^n }{dx^n}(_{a}I^{n-\alpha}f)(x)=\frac{1}{\Gamma(n-\alpha)}\frac{d^n}{dx^n}\int_a^x (x-t)^{n-\alpha-1}f(t)dt,~x>a.
\end{equation}
and
\begin{equation}\label{b2}
 ( D_b^{\alpha}f)(x)= (-1)^n\frac{d^n }{dx^n}(I_b^{n-\alpha}f)(x)=\frac{(-1)^n}{\Gamma(n-\alpha)}\frac{d^n}{dx^n}\int_x^b (t-x)^{n-\alpha-1}f(t)dt,~x<b.
\end{equation}
Here, $n=[Re(\alpha)]+1$, where $[x]$ means the greatest integer   less than or equal to $x$.
The  left and right Caputo fractional derivatives are defined by
\begin{equation}\label{Cap}
  (_{a}^{C}D^{\alpha}f)(x)=(_{a}D^{\alpha }(f(t)-\sum_{k=0}^{n-1}\frac{f^{(k)}(a)}{k!}(t-a)^k)(x)
\end{equation}

and
\begin{equation}\label{Capr}
  (^{C}D_b^{\alpha}f)(x)=(D_b^{\alpha }(f(t)-\sum_{k=0}^{n-1}\frac{f^{(k)}(b)}{k!}(b-t)^k)(x)
\end{equation}
In the space of $n$-th order absolutely continuous functions the definitions (\ref{Cap}) and (\ref{Capr}) can be represented by the modified versions of the following Caputo fractional derivatives:
\begin{equation}\label{mCap}
  (_{a}^{C}D^{\alpha}f)(x)= (~_{a}I^{n-\alpha} f^{(n)})(x)
\end{equation}
called (modified left Caputo fractional derivative )
and
\begin{equation}\label{mCapr}
  (^{C}D_b^{\alpha}f)(x)=(I_b^{n-\alpha} (-1)^nf^{(n)})(x)
\end{equation}
called (modified right Caputo fractional derivative ).

For the conformable fractional calculus concepts and notations we refer to \cite{Roshdi} and \cite{Thabet conf}.
The integrating by parts plays an important role in proving these modified representations.
\section{Main Results: Multiplicative Riemann-Liouville Fractional Integrals and Derivatives} \label{s:3}

Let us define the multiplicative left type iterated integrals $( ~_{a}I_*^nf)(x) = \int_a^x (_{a}I_*^{n-1}f)(t)_{dt}=e^{\int_a^x ln((_{a}I_*^{n-1}f)(t))dt }, n=1,2,3,...$, where $(_{a}I_*^1f)(x)=\int_a^x f(t)_{dt}=e^{\int_a^x (ln\circ f)(t)dt}=e^{(_{a}I^1(ln\circ f)(x)}$ and $(_{a}I_*^0 f)(x)=f(x)$. Similarly, define the right ones as
$( ~_{*}I_b^nf)(x) = \int_x^b ({*}I_b^{n-1}f)(t)_{dt}=e^{\int_x^b ln((I_b^{n-1}f)(t))dt }, n=1,2,3,...$, where $(_{*}I_b^1f)(x)=\int_x^b f(t)_{dt}=e^{\int_x^b (ln\circ f)(t)dt}$ and $(_{*}I_b^0 f)(x)=f(x)$.
If we proceed inductively for $n$, we can easily see that
\begin{equation}\label{e1}
  (~_{a}I_*^nf)(x)=e^{(~_{a}I^n(ln\circ f))(x)}.
\end{equation}
and
\begin{equation}\label{e2}
( _{*}I_b^nf)(x)=e^{ (I_b^n(ln \circ f))(x)}.
\end{equation}
Hence, we can define left and right multiplicative Riemann-Liouville fractional integrals as follows.
\begin{definition}
The multiplicative left Riemann-Liouville fractional integral $(~_{a}I_*^{\alpha}f(x))$ of order $\alpha \in \mathbb{C}$, $Re(\alpha)>0$ starting from $a$ is defined by
\begin{equation}\label{m1}
  (~_{a}I_*^{\alpha}f)(x)=e^{(~_{a}I^{\alpha}(ln\circ f))(x)}.
\end{equation}
and the multiplicative right one is defined by
\begin{equation}\label{m2}
  ( _{*}I_b^{\alpha}f)(x)=e^{ (I_b^{\alpha}(ln \circ f))(x)}
\end{equation}
\end{definition}
Now, we can define the left and right multiplicative Riemann-Liouville fractional derivatives $(_{a}D_*^{\alpha} f)(x)$ and $(_{*}D_b^{\alpha} f)(x)$ as follows.

\begin{definition}The left multiplicative  Riemann-Liouville fractional derivative for $f(x)$ of order $\alpha \in \mathbb{C}$, $Re(\alpha)>0$ starting from $a$ is defined by
\begin{eqnarray}
\nonumber
   (_{a}D_*^{\alpha} f)(x) &=& (\frac{d_*}{dx_*})^n (_{a}I_*^{n-\alpha}f)(x)=e^{ (\frac{d}{dx})^n ln(~_{a}I_*^{n-\alpha}f)(x) ) }\\
   &=& e^{  (\frac{d}{dx})^n (~_{a}I^{n-\alpha} (ln\circ f) (x)}= e^{(~_{a}D^{\alpha}(ln\circ f)(x)  }
\end{eqnarray}
and the right one is defined by
\begin{eqnarray}
\nonumber
   (_{*}D_b^{\alpha} f)(x) &=& (\ominus \frac{d_*}{dx_*})^n (_{*}I_b^{n-\alpha}f)(x)=e^{ (-\frac{d}{dx})^n ln(~_{*}I_b^{n-\alpha}f)(x) ) }\\
   &=& e^{  (-\frac{d}{dx})^n (~I_b^{n-\alpha} (ln\circ f) (x)}= e^{(D_b^{\alpha}(ln\circ f)(x)  }
\end{eqnarray}
\end{definition}
In the next example we find the mutliplicative Riemann-Liouville fractional integrals and derivatives Riemann-Liouville for certain positive functions.
\begin{example}
If $Re(\alpha)\geq 0$ and $\beta \in \mathbb{C}$ ($Re(\beta)>0$), then\begin{eqnarray}
                                                                       ( _{a}I_*^{\alpha} e^{(t-a)^{\beta-1}})(x) &=& e^{\frac{\Gamma(\beta)}{\Gamma(\alpha+\beta)}(x-a)^{\alpha+\beta+1}}, ~~Re(\alpha)> 0\\
                                                                        ( _{a}D_*^{\alpha} e^{(t-a)^{\beta-1}})(x) &=& e^{\frac{\Gamma(\beta)}{\Gamma(\alpha-\beta)}(x-a)^{\alpha-\beta+1}},~~Re(\alpha)\geq 0 \\
                                                                        ( _{*}I_b^{\alpha} e^{(b-t)^{\beta-1}})(x) &=& e^{\frac{\Gamma(\beta)}{\Gamma(\alpha+\beta)}(b-x)^{\alpha+\beta+1}},~~Re(\alpha)> 0 \\
                                                                         ( _{*}D_b^{\alpha} e^{(b-t)^{\beta-1}})(x) &=& e^{\frac{\Gamma(\beta)}{\Gamma(\alpha-\beta)}(b-x)^{\alpha-\beta+1}},~~Re(\alpha)\geq 0
                                                                      \end{eqnarray}
\end{example}
One can note that the multiplicative derivative for the constant function is $1$. However, it is not the case for Riemann-Liouville multiplicative fractional derivative case. Indeed, for example
\begin{equation}\label{ccc}
  ( _{a}D_*^{\alpha} e)(x) = e^{\frac{1}{\Gamma(1-\alpha)}(x-a)^{-\alpha}},~~Re(\alpha)\geq 0
\end{equation}
It is easy to verify that the multiplicative Riemann-Liouville fractional derivative behaves well in case of multiplication of functions. Indeed,
\begin{equation}\label{x}
  ( _{a}D_*^{\alpha} fg)(x)=( _{a}D_*^{\alpha}f)(x).( _{a}D_*^{\alpha} g)(x).
\end{equation}
\section{Main Results: Multiplicative Letnikov Fractional Integrals and Derivatives} \label{s:4}
Let us consider a continuous positive function $f(x)$. From the definition of the (backward) multiplicative derivative we have
\begin{equation}\label{one}
  f_*(t)=\lim_{h\rightarrow 0}(\frac{f(t)}{f(t-h)})^{1/h}
\end{equation}
Applying this definition twice gives the second order multiplicative derivative:
\begin{equation}\label{two}
  f_*^{(2)}(t)=\lim_{h\rightarrow 0}(\frac{f(t)f(t-2h)}{f(t-h)^2})^{1/h^2}.
\end{equation}
If we proceed inductively, we reach at
\begin{equation}\label{three}
 f_*^{(n)}(t)=\lim_{h\rightarrow 0}~_{*}f_h^{(n)},
\end{equation}
where
\begin{equation}\label{four}
  _{*}f_h^{(n)}=\left( \prod_{r=0}^n f(t-rh)^{(-1)^r \left(
                                                  \begin{array}{c}
                                                    n \\
                                                    r \\
                                                  \end{array}
                                                \right)
  }\right)^{\frac{1}{h^n}},
\end{equation}
and $ \left(
                                                  \begin{array}{c}
                                                    n \\
                                                    r \\
                                                  \end{array}
                                                \right)
  =\frac{n(n-1)...(n-r+1)}{r! }$ is the usual for the binomial coefficient.
  Let us now consider the following generalizing expression:
  \begin{equation}\label{gen}
    _{*}f_h^{(p)}=\left(\prod_{r=0}^n f(t-rh)^{(-1)^r \left(
                                                  \begin{array}{c}
                                                    p \\
                                                    r \\
                                                  \end{array}
                                                \right)
  }\right)^{\frac{1}{h^p}},
  \end{equation}
  where $p$ is an arbitrary integer.
  Notice that
  \begin{equation}\label{not}
   _{*}f_h^{(p)}=e^{g_h^{(p)}},
  \end{equation}
  where $g(t)=ln(f(t))$ and $g_h^{(p)}=\frac{1}{h^p} \sum_{r=0}^n g(t-rh)(-1)^r \left(
                                                  \begin{array}{c}
                                                    p \\
                                                    r \\
                                                  \end{array}
                                                \right)$ is the fraction used to obtain the usual Letnikov fractional integrals and derivatives (see \cite{Pod}, sec. 2.2.).
  Obviously, for $p\leq n$, we have

  \begin{equation}\label{clear}
    lim_{h\rightarrow 0} ~ _{*}f_h^{(p)}(t)=f_*^{(p)}(t)
  \end{equation}

  To define the Letnikov fractional integrals, let us consider the negative values of $p$. For this purpose denote

  \begin{equation}\label{den}
    [\begin{array}{c}
       p \\
       r
     \end{array}
    ]=\frac{p(p+1)...(p+r-1)}{r!}.
  \end{equation}
  Then, clearly we have
  \begin{equation}\label{cl}
  (\begin{array}{c}
       -p \\
       r
     \end{array}
  )=(-1)^r[\begin{array}{c}
       p \\
       r
     \end{array}
    ].
  \end{equation}
  Thus, we can express
  \begin{equation}\label{exp}
    _{*}f_h^{(-p)}(t)=\left(\prod_{r=0}^n f(t-rh)^{\left [
                                                  \begin{array}{c}
                                                    p \\
                                                    r \\
                                                  \end{array}
                                                \right]
  }\right)^{h^p},
  \end{equation}
  If $n$ is fixed, then $_{*}f_h^{(-p)}$  tends to the uninteresting limit $1$ as $h\rightarrow 0$. To arrive to a nonzero limit, we have to suppose that $n\rightarrow \infty$ as $h\rightarrow 0$. Thus we can take $h=\frac{t-a}{n}$ and then denote
  \begin{equation}\label{andd}
    \lim_{h\rightarrow 0, nh=t-a} ~_{*}f_h^{(-p)}(t)=(~_{a}D_*^{-p}f)(t)
  \end{equation}
  Now making use of (\ref{not}) and proceeding as in \cite{Pod} (Chapter 2), we come to the representation:
  \begin{equation}\label{repr}
    (~_{a}D_*^{-p}f)(t)=(~_{a}I_*^p f)(t)
  \end{equation}
  Finally, we give the definition of multiplicative Letnikov fractional derivatives.

  \begin{definition}
  For $\alpha>0$, we define the left and right multiplicative Letnikov fractional derivatives, respectively by:
  \begin{equation}\label{the left}
   ( _{a}D_*^{\alpha)}f)(t)=\lim_{h\rightarrow 0^+} \left(   \prod_{r=0}^{[\frac{t-a}{h}]} f(t-rh)^{(-1)^r \left(
                                                                                                            \begin{array}{c}
                                                                                                              \alpha \\
                                                                                                              r \\
                                                                                                            \end{array}
                                                                                                          \right)
    }
\right)^{\frac{1}{h^{\alpha}} }
    \end{equation}
    and
    \begin{equation}\label{the right}
      ( _{*}D_b^{\alpha)}f)(t)=\lim_{h\rightarrow 0^+} \left(   \prod_{r=0}^{[\frac{b-t}{h}]} f(t+rh)^{(-1)^r \left(
                                                                                                            \begin{array}{c}
                                                                                                              \alpha \\
                                                                                                              r \\
                                                                                                            \end{array}
                                                                                                          \right)
    }
\right)^{\frac{1}{h^{\alpha}} }.
    \end{equation}
  \end{definition}
Where, $\left(
                                                                                                            \begin{array}{c}
                                                                                                              \alpha \\
                                                                                                              r \\
                                                                                                            \end{array}
                                                                                                          \right)=\frac{\Gamma(\alpha+1)} {\Gamma(\alpha-r+1).\Gamma (r+1)}$.
\section{Main Results: Multiplicative Caputo Fractional Derivatives} \label{s:4}
Its well-known that the multiplicative derivative of the constant function is $1$. In Section 2, we noticed that the multiplicative fractional Riemann-Liouville derivative for the constant function was not $1$. In this section we shall define the Caputo multiplicative fractional derivative with the property that it assigns the constant function to $1$.

\begin{definition} \label{xx} Let $\alpha \in \mathbb{C}$ such that $Re(\alpha)>0$. Let $n=[Re(\alpha)]+1$. Then the left (modified)Caputo multiplicative fractional derivative of the positive real valued function $f$ is defined by

\begin{equation}\label{mod}
 (_{a}^{C}D_*^{\alpha}f)(t)=(_{a}I^{n-\alpha}_* f_*^{(n)})(t)
\end{equation}
\end{definition}
 From Definition \ref{xx}, the definition of multiplicative fractional integrals and the definition of multiplicative derivative, we can easily see that

 \begin{eqnarray}
 \nonumber
   (~_{a}^{C}D_*~^{\alpha}f)(t) &=& (~_{a}I^{n-\alpha} _* f_* ^{(n)})(t) \\ \nonumber
   &=& e^{(~_{a}I^{n-\alpha}ln\circ f_*~^{(n)}})(t)  = e^{ (~_{a}I^{n-\alpha}ln(e^{\frac{d^n}{dx^n} ln\circ f}) } (t) \\
    &=& e^{ (~_{a}I^{n-\alpha} \frac{d^n}{dx^n}ln\circ f })(t)=  e^{(~_{a}^{C}D^{\alpha} ln \circ f)(t)}
 \end{eqnarray}
Similarly, in the right case we define the right (modified) multiplicative Caputo fractional derivative by $ (~_{*}^{C}D_b~^{\alpha}f)(t) = (~_{*}I^{n-\alpha} _b (\ominus f_*)^{(n)})(t)$ and we can show that
\begin{equation}\label{we can}
  (~_{*}^{C}D_b~^{\alpha}f)(t)=e^{(^{C}D_b^{\alpha} ln \circ f)(t)}
\end{equation}
Thus, the (modified) Caputo multiplicative fractional derivative also behaves like the multiplicative  Riemann one and usual multiplicative one by being lifted to the exponential and logarithm functions.

Finally, we give analogous multiplicative  representations to (\ref{Cap}) and (\ref{Capr}).
From (\ref{we can}) and (\ref{Cap}) applied to $g(t)=ln(f(t))$ we have
\begin{eqnarray}
\nonumber
 (_{a}^{C}D_*~^{\alpha}f)(t) &=&  e^{(~_{a}^{C}D^{\alpha} ln \circ f)(t)} \\ \nonumber
  &=&e^{  ~_{a}D ^ {\alpha}\left[  g(t)-\sum_{k=0}^{n-1}\frac{g^{(k)}(a)}{k!}(t-a)^k\right]} \\ \nonumber
   &=& e^{(~_{a}D ^ {\alpha}g)(t)}.e^{-(~_{a}D ^ {\alpha}S_g)(t)}=e^{(~_{a}D ^ {\alpha}g)(t)}.e^{(~_{a}D ^ {\alpha} ln\circ e^{-S_g})(t)}\\ \nonumber
   &=& ~_{a}D_* ^{\alpha} (f(t)e^{-S_g(t)}),
\end{eqnarray}
where $S_g(t)=\sum_{k=0}^{n-1}\frac{g^{(k)}(a)}{k!}(t-a)^k$. Also, notice that
\begin{equation}\label{finall}
  e^{-S_g(t)}=\prod_{k=0}^{n-1} e^{-\frac{g^{(k)}(a)}{k!}(t-a)^k}
\end{equation}
Next theorem describes the action of the multiplicative fractional integral to multiplicative (modified ) Caputo fractional derivative.
\begin{theorem} Let $f$ be a positive real-valued function, $g(t)=ln(g(t))$ and let $n=[Re(\alpha)]+1$. Then
\begin{itemize}
  \item a)$(~_{a}I_*^{\alpha} ~_{a}^{C}D_*^{\alpha}f)(t)=\frac{f(t)}{\prod_{k=0}^{n-1} e^{ \frac{g^{(k)}(a)(t-a)^k}{k!}  }}$, $t>a$.
  \item b) $(~_{*}I_b^{\alpha} ~_{*}^{C}D_b^{\alpha}f)(t)=\frac{f(t)}{\prod_{k=0}^{n-1} e^{ (-1)^k\frac{g^{(k)}(b)(b-t)^k}{k!}  }}$, $t<b $.
\end{itemize}
\end{theorem}

\begin{proof}
The proof of (a) follows by definition and that $(~_{a}I^{\alpha} ~_{a}^{C}D^{\alpha}g)(t)=g(t)-\sum_{k=0}^{n-1} \frac{g^{(k)}(a)(t-a)^k }{k!}$. The proof of (b) follows by definition and that $(I_b^{\alpha} ~^{C}D_b^{\alpha}g)(t)=g(t)-\sum_{k=0}^{n-1} \frac{(-1)^k g^{(k)}(b)(b-t)^k}{k!}$.
\end{proof}
\section{The multiplicative local conformable fractional derivative}

Finally, we remark the multiplicative local (conformable) fractional derivatives and integrals. For local (confromable) fractional derivative we refer to \cite{Roshdi} and \cite{Thabet conf}.

For a positive real-valued function $f$ defined on $[a,b]$, we define the multiplicative left and right conformable fractional derivatives of order $0<\alpha \leq 1$ starting from $a$ and ending at $b$, respectively, by
\begin{equation}\label{confleft}
  (^{*}T_{\alpha}^af)(t)=\lim_{\epsilon\rightarrow 0}\left (\frac{f(t+\epsilon (t-a)^{1-\alpha})}{f(t)}\right)^{1/\epsilon}
\end{equation}
and
\begin{equation}\label{confright}
  (_{\alpha}^{b}T^*f)(t)=\lim_{\epsilon\rightarrow 0}\left (\frac{f(t+\epsilon (b-t)^{1-\alpha})}{f(t)}\right)^{-1/\epsilon}
\end{equation}
Applying the logarithm to both sides in (\ref{confleft}) and (\ref{confright}),  using the fact that $(T_{\alpha}^af)(t)=(t-a)^{1-\alpha}f^\prime (t)$ and $(_{\alpha}^{b}Tf)(t)=-(b-t)^{1-a\alpha}f^\prime(t)$ \cite{Thabet conf}, with an application of L'Hopital rule, we conclude that $(^{*}T_{\alpha}^af)(t)=e^{(T_{\alpha}^aln\circ f)(t)}=e^{\frac{(T_{\alpha}^af)(t)}{f(t)}}$ and $(_{\alpha}^{b}T^*f)(t)=e^{(_{\alpha}^{b}Tln\circ f)(t)}=e^{\frac{(_{\alpha}^{b}Tf)(t)} {f(t)}}$.

For $0< \alpha \leq 1 $ and a positive real-valued function $f$ the left and right multiplicative conformable fractional integrals are defined are defined by

\begin{equation}\label{m conf integ l}
(~^{*}I_{\alpha}^a f)(t)=\int_a^t f(x)_{d^*_{\alpha}(x,a)}=\int_a^t f(x)^{(x-a)^{\alpha-1}}_{dx}=e^{\int_a^t (x-a)^{\alpha-1}ln(f(x))dx}
\end{equation}

and

\begin{equation}\label{m conf integ r}
(~^{b}_{\alpha}I^* f)(t)=\int_t^b f(x)_{d^*_{\alpha}(b,x)}=\int_t^b f(x)^{(b-x)^{\alpha-1}}_{dx}=e^{\int_t^b (b-x)^{\alpha-1}ln(f(x))dx},
\end{equation}
respectively.

To define the higher order multiplicative conformable fractional derivative and integrals we proceed as in \cite{Thabet conf}. Indeed, for $\alpha \in (n,n+1]$ and $\beta=\alpha-n$, we define

\begin{equation}\label{aa}
  (^{*}\textbf{T}_{\alpha}^af)(t)=(^{*}T_{\beta}^af_*^{(n)})(t)=e^{T_{\beta}^a ln (f_*^{(n)}(t) )}=e^{T_{\beta}^a \frac{d^n}{dt^n}ln(f(t))}.
\end{equation}

\begin{equation}\label{bb}
  (^{b}_{\alpha}\textbf{T}^*f)(t)=(~^{b}_{\beta}T^* f_*^{(n)})(t)=e^{~_{\beta}^{b}T ln (f_*^{(n)}(t) )}=e^{~_{\beta}^{b} T \frac{d^n}{dt^n}ln(f(t))}.
\end{equation}

For higher order integrals we define

\begin{equation}\label{cc}
  (~^{*}I_{\alpha}^a f)(t)=~_{a}I_*^{n+1}(f(t)^{(t-a)^{\beta-1}}).
\end{equation}

\begin{equation}\label{dd}
  (~^{b}_{\alpha}I^*f)(t)=~_{*}I_b^{n+1}(f(t)^{(b-t)^{\beta-1}}).
\end{equation}
Notice that when $n=0$ or $0<\alpha \leq 1$, the definitions coincide with the previous ones.

To confirm our definitions, we state the following proposition describing the action of the multiplicative conformable fractional derivative and integral operators on each other. This result will be useful in solving multiplicative conformable fractional dynamical systems by transforming them to multiplicative integral ones.
\begin{proposition} \label{above} Let $f$ be a positive real-valued function defined on $[a,b]$ and $0< \alpha \leq 1$. Then
\begin{itemize}
  \item a) $(^{*}T_{\alpha}^a~^{*}I_{\alpha}^a f)(t)=f(t)$, for $f$ is continuous.
  \item b)$(_{\alpha}^{b}T^*~^{b}_{\alpha}I^* f)(t)=f(t)$, for $f$ is continuous.
  \item c)$(I_{\alpha}^a ~^{*}T_{\alpha}^a~^{*} f)(t)=\frac{f(t)}{f(a)}$.
  \item d)$(~^{b}_{\alpha}I^*~ _{\alpha}^{b}T^*f)(t)=\frac{f(t)}{f(b)}$.
\end{itemize}
\end{proposition}
\begin{proof}
We prove (a) and (c). Parts (b) and (d) are similar.
\begin{itemize}
  \item \textbf{The proof of (a)}: By definition and the help of the fact that $(T_{\alpha}^a~I_{\alpha}^a g)(t)=g(t)$ (see Lemma 2.1 in \cite{Thabet conf}) we have

      \begin{eqnarray}
       \nonumber
        (^{*}T_{\alpha}^a~^{*}I_{\alpha}^a f)(t) &=& e^{T_{\alpha}^a ln(^{*}I_{\alpha}^af(t))}= e^{T_{\alpha}^a \int_a^t (x-a)^{\alpha-1} ln(f(x))dx}\\
         &=&  e^{T_{\alpha}^a I_{\alpha}^a(ln(f(t))}=e^{ln(f(t))}=f(t).
      \end{eqnarray}
  \item \textbf{The proof of (c)}: By definition and the help of Lemma 2.8 in \cite{Thabet conf}, we have

  \begin{eqnarray}
   \nonumber
     (~^{*}I_{\alpha}^a~^{*}T_{\alpha}^a f)(t) &=& e^{\int_a^t (x-a)^{\alpha-1}ln(^{*}T_{\alpha}^a f)(x)dx}= e^{\int_a^t (x-a)^{\alpha-1}T_{\alpha}^a ln(f(x)) dx}\\
    e^{\int_a^t (x-a)^{\alpha-1} (x-a)^{1-\alpha} \frac{f^\prime(x)}{f(x)} dx} &=&  e^{ln(f(t))-ln(f(a))}=\frac{f(t)}{f(a)}.
  \end{eqnarray}
\end{itemize}
\end{proof}

Proposition \ref{above} above can be generalized for the higher order case. For example parts (a) and (b) are still true for arbitrary $\alpha$ if we noticed that $~_{a}I^1_*(f(x)^{(x-a)^{\beta-1}})(t)=(~^{*}I_{\beta}^a f)(t)$ and $~^{*}I_b^1(f(x)^{(b-x)^{\beta-1}})(t)=(~_{b} ^{\beta}I^* f)(t)$.

Next theorem will generalize the parts (b) and (d) of Proposition \ref{above}.
\begin{theorem} \label{finalt}
Assume $\alpha \in (n,n+1]$ , $\beta=\alpha-n$, and $f$ a positive real-valued function then which is $n+1$ times differentiable.Then
\begin{itemize}
  \item a)$(~^{*}I_{\alpha}^a ~^{*}\textbf{T}_{\alpha}^a f)(t)= \frac{f(t)}  { \prod_{k=0}^n e^{ \frac{g^{(k)}(a) (t-a)^k} {k!}} }$,
  \item b)$(~^{b}_{\alpha}I^* ~^{*}\textbf{T}_{\alpha}^a f)(t)= \frac{f(t)}  { \prod_{k=0}^n e^{ (-1)^k\frac{g^{(k)}(b) (b-t)^k} {k!}} }$,
\end{itemize}
where $g(t)=ln(f(t))$.
\end{theorem}
\begin{proof}
We prove (a) only. The other part is similar.
\begin{eqnarray}
\nonumber
  (~^{*}I_{\alpha}^a ~^{*}\textbf{T}_{\alpha}^a f)(t)&=& ~_{a}I_*^{n+1} ( [~^{*}\textbf{T}_{\alpha}^a f(x)]^{(x-a)^{(\beta-1)}})(t)= ~_{a}I_*^{n+1} ( [~^{*}T_{\beta}^a f_*^{(n)}(x)]^{(x-a)^{(\beta-1)}})(t)\\
   &=& e^{(~_{a}I^{n+1}(x-a)^{\beta-1}ln(~^{*}T_{\beta}^af_*^{(n)}(x))(t)}= e^{(~_{a}I^{n+1}(x-a)^{\beta-1} T_{\beta}^a \frac{d^n ln(f(x))}{dx^n})(t)}\\ \nonumber
   &=& e^{(~_{a}I^{n+1}(x-a)^{\beta-1} (x-a)^{1-\beta} \frac{d^{n+1} ln(f(x))} {dx^{n+1}} )(t)} \\ \nonumber
   &=& e^{g(t)-\sum_{k=0}^n \frac{ g^{(k)}(a) (t-a)^k } {k!} }=\frac{f(t)}  { \prod_{k=0}^n e^{ \frac{g^{(k)}(a) (t-a)^k} {k!}} }
\end{eqnarray}
\end{proof}

\end{document}